\documentclass[a4paper,11pt,reqno]{amsart}

\usepackage[utf8]{inputenc}
\usepackage[T1]{fontenc}
\usepackage{lmodern}
\usepackage[english]{babel}
\usepackage{microtype}

\usepackage{amsmath,amssymb,amsfonts,amsthm,esint}
\usepackage{mathtools,accents}
\usepackage{mathrsfs}
\usepackage{aliascnt}
\usepackage{braket}
\usepackage{bm}

\usepackage[a4paper,margin=3cm]{geometry}
\usepackage[citecolor=blue,colorlinks]{hyperref}

\usepackage{enumerate}
\usepackage{xcolor}

\makeatletter
\g@addto@macro\@floatboxreset\centering
\makeatother

\makeatletter
\def\newaliasedtheorem#1[#2]#3{
	\newaliascnt{#1@alt}{#2}
	\newtheorem{#1}[#1@alt]{#3}
	\expandafter\newcommand\csname #1@altname\endcsname{#3}
}
\makeatother

\numberwithin{equation}{section}

\newtheoremstyle{slanted}{\topsep}{\topsep}{\slshape}{}{\bfseries}{.}{.5em}{}

\theoremstyle{plain}
\newtheorem{theorem}{Theorem}[section]
\newaliasedtheorem{proposition}[theorem]{Proposition}
\newaliasedtheorem{lemma}[theorem]{Lemma}
\newaliasedtheorem{corollary}[theorem]{Corollary}
\newaliasedtheorem{counterexample}[theorem]{Counterexample}

\theoremstyle{definition}
\newaliasedtheorem{definition}[theorem]{Definition}
\newaliasedtheorem{question}[theorem]{Question}
\newaliasedtheorem{openquestion}[theorem]{Open Question}
\newaliasedtheorem{conjecture}[theorem]{Conjecture}

\theoremstyle{remark}
\newaliasedtheorem{remark}[theorem]{Remark}
\newaliasedtheorem{example}[theorem]{Example}

\newcommand{\setN}{\mathbb{N}}

\newcommand{\setR}{\mathbb{R}}

\newcommand{\setZ}{\mathbb{Z}}

\let\altphi\phi
\let\phi\varphi
\let\varphi\altphi
\let\altphi\undefined

\newcommand{\diam}{{\rm diam}}

\DeclareMathOperator{\Ric}{Ric}

\newcommand{\dist}{\mathsf{d}}

\newfont{\tmpf}{cmsy10 scaled 2500}

\def\XXint#1#2#3{{\setbox0=\hbox{$#1{#2#3}{\int}$ }
		\vcenter{\hbox{$#2#3$ }}\kern-.6\wd0}}

\begin{document}

\title{Stability of Tori under Lower Sectional Curvature}

\author{Elia Bru\`e, Aaron Naber and Daniele Semola}

\maketitle

\begin{abstract}
	Let $(M^n_i, g_i)\stackrel{GH}{\longrightarrow} (X,\dist_X)$ be a Gromov-Hausdorff converging sequence of Riemannian manifolds with ${\rm Sec}_{g_i} \ge -1$,  ${\rm diam}\, (M_i)\le D$, and such that the $M^n_i$ are all homeomorphic to tori $T^n$.  Then $X$ is homeomorphic to a $k$-dimensional torus $T^k$ for some $0\leq k\leq n$.  This answers a question of Petrunin in the affirmative.  We show this result is false if the $M^n_i$ are homeomorphic tori which are only assumed to be Alexandrov spaces.  When $n=3$, we prove the same tori stability under the weaker condition ${\rm Ric}_{g_i} \ge -2$. 
\end{abstract}

\section{Introduction}

The study of collapsing Riemannian manifolds under different curvature constraints is a subject of profound interest \cite{Gromovalmostflat,FukayacollapsingI,FukayacollapsingII,Fukayaboundary,CheegerGromovI,CheegerGromovII,Fukayaorbifold,Yamaguchicoll,CheegerFukayaGromov} in geometric analysis.

\bigskip

In this paper we focus our attention on sequences of manifolds $(M_i^n,g_i)$ with uniform lower bounds on the sectional curvature and diameter bounds:

\begin{equation}\label{H}
	\begin{split}
		{\rm Sec}_{g_i} & \ge -1\, ,
		\\
		\diam (M_i) & \le D \, .
	\end{split}
\end{equation}
In this setting, Gromov-Hausdorff limits $(M^n_i, g_i)\stackrel{GH}{\longrightarrow} (X^k,\dist)$ are Alexandrov spaces of dimension $0\le k \le n$, curvature $\ge -1$ , and $\diam(X)\le D$.  It is also well-known that in many situations the topology of the limit space $(X^k,\dist)$ is closely tied to the topology of the approximating manifolds $M_i^n$.\\ 

In the non-collapsing setting, i.e. when $k=n$, Perelman proved that $X$ is homeomorphic to $M_i$ for $i$ large enough \cite{Perelmanstability}.
When $k<n$ and $X^k$ is a smooth Riemannian manifold without boundary, Yamaguchi \cite{Yamaguchicoll} proved for large $i$ that $M_i^n$ is the total space of a locally trivial fiber bundle with base space $X^k$.\\

However, the topology of collapsing sequences is not fully understood when the limits are singular or admit boundary points, even though some general local restrictions have been obtained by Kapovitch in \cite{Kapovitchcollapse}.  The main result of this paper establishes that when $M^n_i$ all are homeomorphic to tori $T^n$, then the limit must be homeomorphic to a torus, regardless of any additional assumption on its regularity.\\

\begin{theorem}\label{thm:torusstabilitysecintro}
Let $(M^n_i, g_i) \xrightarrow{GH} (X^k,\dist_X)$ satisfy ${\rm Sec}_{g_i} \ge -1$,  ${\rm diam}\, (M_i)\le D$, and such that the $M^n_i$ are all homeomorphic to tori $T^n$.  Then $X^k$ is homeomorphic to the $k$-dimensional torus $T^k$ for some $0\le k\le n$.
\end{theorem}

We can extend or provide counterexamples to extensions of the above in various contexts.  If we just assume that $(M^n_i,\dist_i)$ are all Alexandrov spaces homeomorphic to $T^n$ with ${\rm curv} \ge -1$ and  ${\rm diam}\, (M_i)\le D$, then for $n\leq 4$ we can still prove that $X^k$ is homeomorphic to $T^k$ and we fully recover the conclusion of \autoref{thm:torusstabilitysecintro}.  For $n\geq 5$ the situation becomes more subtle.  We can show in this case that $X^k$ is at least homotopic to a torus $T^k$, and an example suggested to the authors by Vitali Kapovitch shows that this is sharp.  In Section \autoref{subsec:counterAlexandrov} we will show all of this, and in particular provide an example of Alexandrov spaces $M^5_i\to X^4$ where the $M^5_j$ are all homeomorphic tori but the limit $X^4$ is not.\\

The framework presented in \autoref{thm:torusstabilitysecintro} has garnered recent attention, as Zamora in \cite{Zamora} studied sequences of tori $(M_i^n, g_i)$ that satisfy condition \eqref{H} and demonstrated that they cannot converge to an interval $[0, L]$, thus covering the case when $n$ is arbitrary and $k=1$ in our main result. The very same result had been previously established by Katz in \cite{Katz20}, with a more elementary argument, under the additional assumption $n=2$. We also remark that the case when $n=3$ in our main result could be handled with the general theory of collapse under a sectional curvature lower bound in dimension 3, as developed by Shioya and Yamaguchi in \cite{ShioyaYamaguchi}.\footnote{See the discussion in https://mathoverflow.net/questions/236001/gromov-hausdorff-limits-of-2-dimensional-riemannian-surfaces} Furthermore, in the same paper, Zamora mentions the conjecture, due to Petrunin (see also \cite{Zamora2}), that any limit in this setting must necessarily be homeomorphic to a torus. Our result, as presented in \autoref{thm:torusstabilitysecintro}, confirms this conjecture.

\bigskip

Let us briefly comment the interplay between the topological assumption $M^n_i \cong T^n$ and the metric assumption \eqref{H} in \autoref{thm:torusstabilitysecintro}.

First, if we strengthen the sectional curvature condition to ${\rm Sec}_{g_i} \ge 0$, a rigidity phenomenon occurs: any torus in the sequence must be flat.  For this rigidity it is actually enough to assume ${\rm Scal}_{g_i} \ge 0$, as proven first by Schoen and Yau when $n\le 7$ \cite{SchoenYau}, and then in general dimensions by Gromov and Lawson \cite{GromovLawson}. A Gromov-Hausdorff limit of a sequence of flat tori must be a flat torus, by Mahler's compactness theorem \cite{Mahler46}.
When nonnegative sectional is weakened to almost nonnegative Ricci curvature, then an ``almost rigidity'' occurs and the universal cover is almost Euclidean (see \cite{Colding}).

However, under the general lower curvature bound condition \eqref{H}, we do not encounter the same rigidity as described above. This, in fact, constitutes the primary novel aspect of our theorem: the interaction between the topological assumption $M^n_i \cong T^n$ and the lower bound on the sectional curvature still controls the topology of limits, even though we are out of the rigidity regime.

In essence, this is possible because $M^n_i \cong T^n$ implies uniform volume non-collapsing of the universal covers $(\tilde M_i^n, \tilde g_i)$. Consequently, by selecting a suitable intermediate covering, we can apply Perelman's stability theorem, which ensures a uniform control over the topology of the universal covers.
Then, a careful argument is needed to gain uniform control on the action of $\pi_1(M_i)$ on $\tilde M_i\cong \mathbb{R}^n$. The application of Perelman's stability theorem mentioned above (and of a subsequent refinement due to Kapovitch \cite{Kapovitchlimits}), are the main steps where the lower sectional curvature bound plays a key role. We believe that the remaining parts of the argument could be adapted to the case of lower Ricci curvature bounds

\bigskip

We finally pose an open question about sequence of tori satisfying a lower bound on the Ricci curvature:
\begin{equation}\label{H2}
	\begin{split}
		{\rm Ric}_{g_i} & \ge -(n-1)\, ,
		\\
		\diam (M_i) & \le D \, .
	\end{split}
\end{equation}

\begin{question}\label{question}
	Let $(M^n_i, g_i) \xrightarrow{GH} (X^k,\dist_X)$ satisfy \eqref{H2}, where $k$ is the rectifiable dimension of $X$. Assume that $M^n_i$ are all homeomorphic to tori $T^n$. Is $X^k$ homeomorphic to $T^k$?	
\end{question}

The answer to the above question is affirmative when $n=3$.

\begin{theorem}\label{thm:n=3}
	Let $(M^3_i, g_i) \xrightarrow{GH} (X^k,\dist_X)$ satisfy ${\rm Ric}_{g_i} \ge -2$,  ${\rm diam}\, (M_i)\le D$, and such that the $M^3_i$ are all homeomorphic to tori $T^3$.  Then $X^k$ is homeomorphic to the $k$-dimensional torus $T^k$ for some $0\le k\le 3$.
\end{theorem}

To prove \autoref{thm:n=3}, we can follow the same approach as in \autoref{thm:torusstabilitysecintro} with two major modifications. First, we substitute the application of Perelman's stability theorem with the topological stability results presented in \cite{Simon}. Second, instead of utilizing \autoref{prop:quotienttopman}, we employ general results about Lie group actions of homeomorphisms on topological 3-manifolds from \cite{Bredon2}.  See \autoref{ss:Ricci_lower_bound_n=3} for more details.

\subsection*{Acknowledgment}
The first author would like to express gratitude for the financial support received from Bocconi University.  The last author is supported by FIM-ETH through a Hermann Weyl Instructorship.

The authors are grateful to Sergio Zamora and to the reviewer for carefully reading the note and for useful comments.  They are thankful to Vitali Kapovitch for suggesting the example discussed in \autoref{subsec:counterAlexandrov}.

\vspace{.5cm}

\section{Preliminary Results}

\vspace{.3cm}

\subsection{Noncollapsing of Aspherical Manifolds}

The noncollapsing statement below follows from the combination of $(E_4)$ in \cite[Appendix 1]{Gromovfilling} and \cite[Theorem 4.5.D']{Gromovfilling}. Notice indeed that the universal cover of a Riemannian torus is geometrically contractible, as observed at the beginning of  \cite[Section 4.5.D]{Gromovfilling}.  

\begin{theorem}[Noncollapsing of Aspherical Manifolds \cite{Gromovfilling}]\label{thm:covernoncollapsed}
If $(M^n,g)$ is homeomorphic to a closed aspherical manifold and satisfies ${\rm Ric}_g \ge -(n-1)$, $\diam (X)\le D$, then 
\begin{equation}
{\rm Vol}_g(B_1(\tilde{p}))\ge c(n,D)\, ,\quad \text{for any $\tilde{p}\in \tilde {M}$}\, ,
\end{equation}
where $\tilde M$ is the universal cover of $M$.
\end{theorem}
\begin{remark}
	We also refer the reader to \cite{Guth}, where the lower Ricci assumption was dropped at the expense of a weaker volume conclusion.
\end{remark}

\vspace{.3cm}

\subsection{Lattice Covers of Riemannian Tori}

\begin{definition}[$\ell^1$ Word Norm]
	Let $(M^n,g)$ be homeomorphic to $T^n$, and let $\gamma_1, \ldots , \gamma_n$ a family of generators of $\Gamma\equiv \pi_1(M)$.  We define the associated $\ell^1$ norm on $\Gamma$ by
\begin{equation}
	\| \gamma \|_1 := \sum_{i=1}^n |\alpha_1|\, , 
\end{equation}
where $\gamma=\gamma_1^{\alpha_1} \circ \ldots \circ \gamma_n^{\alpha_n}$.
\end{definition}

Theorem \ref{thm:less dense} below is from \cite[Theorem 4.2]{KS20}.

\begin{theorem}[Uniform Lattice Coverings \cite{KS20}]\label{thm:less dense}
Let $(M^n,g)$ be homeomorphic to $T^n$ with $\diam(M) \le D$. Let $(\tilde M, \tilde g)$ be the universal cover where $\Gamma\equiv \pi_1(M)$ acts by isometries. Then there exists $\Lambda\leq \Gamma$ with $\Lambda \equiv \mathbb{Z}^n$ and 
	\begin{itemize}
		\item[(i)] $\diam(\tilde M/\Lambda) \le 6^n D$;

		\item[(ii)] There exists a family of generators $\{\gamma_1, \ldots , \gamma_n\}$ of $\Lambda$ such that
		\begin{equation}
			A(n)^{-1} D \,\| \gamma \|_1 \le \dist_{\tilde M}(x,\gamma(x)) \le A(n) D\, \| \gamma \|_1 \, , \quad \forall \, \gamma \in \Lambda\, , x\in \tilde M \, .
		\end{equation}
	\end{itemize}
\end{theorem}

\vspace{.3cm}

\subsection{Lie Group actions on Euclidean Space and Tori}

The statement below is a consequence of Newman's theorem. We refer to \cite[Chapter III, Section 9]{Bredon} for a thorough discussion.

\begin{theorem}\label{thm:boundedorbitstrivial}
Let $G$ be a compact Lie group (not necessarily connected). Assume that $G$ acts effectively by homeomorphisms on $\setR^n$. Then either $G=\{e\}$ is trivial or the orbits of $G$ cannot be uniformly bounded with respect to the Euclidean distance. 
\end{theorem}

The following two statements, originally due to \cite{ConnerMontgomery,ConnerRaymond}, are stated as in \cite[Chapter IV, Section 9]{Bredon}.

\begin{theorem}[\cite{ConnerMontgomery,ConnerRaymond}]\label{thm:actionsontori}
Let $G$ be a compact connected Lie group acting effectively by homeomorphisms on a torus $T^n$. Then $G$ is a torus and the action is free. Moreover, the projection to the orbit space is a trivial principal $G$-fibration.
\end{theorem}

\vspace{.3cm}

\subsection{Free Actions on Noncollapsed Limits}

It is a well-known result in Riemannian Geometry that quotients with respect to free isometric group actions of smooth Riemannian manifolds are (smooth Riemannian) manifolds. For the proof of our main result, it is important to have a partial generalization of this statement for non-collapsed Gromov-Hausdorff limits of manifolds with sectional curvature uniformly bounded from below.\\

We recall that an Alexandrov space $(X,\dist_X)$ of dimension $n$ is said to be smoothable if it can be presented as a (non-collapsed) Gromov-Hausdorff limit of a sequence $(M^n_i,g_i)$ of smooth Riemannian manifolds with sectional curvature uniformly bounded from below, see for instance \cite{LebedevaPetrunin}. We will say that an Alexandrov space $(X,\dist_X)$ is locally smoothable if it is locally isometric to a smoothable Alexandrov space.

\begin{proposition}\label{prop:quotienttopman}
Let $(X,\dist_X)$ be a locally smoothable Alexandrov space.
If $G$ is a compact, connected Lie group acting freely by isometries on $(X,\dist_X)$, then $(X/G,\dist_{X/G})$ is a topological manifold.
\end{proposition}

\begin{proof}
The proof relies on the slicing theorem for Alexandrov spaces from \cite{HarveySearle}, and on the characterization of spaces of directions for noncollapsed smooth limits with lower sectional curvature bounds \cite{Kapovitchlimits}. 
\medskip

Let $x\in X$ and let us denote by $\Sigma_xX$ the space of directions at $x$.  In particular the tangent cone at $x$ is given by the cone space $C(\Sigma_x)$.  Let $S_x\subset \Sigma_xX$ be the space of unit directions tangent to the orbit $G\cdot x$ of $x$, and let $\nu_x \equiv \nu(S_x)\subset \Sigma_xX$ be the space of unit directions normal to $S_x$.  Namely,
\begin{equation}
\nu_x:=\large\{v\in \Sigma_xX\, :\, \dist_{\Sigma_xX}(v,w)=\frac{\diam(\Sigma_xX)}{2}\, ,\quad \text{for any $w\in S_x$}\large\}\, .
\end{equation}  

Let $k=\dim G$. By \cite{GalazSearle} we have that $S_x$ is {\it isometric} to a standard $(k-1)$-sphere. Moreover, the space of directions at $x$ is isometric to the join of $S_x$ with its normal space, namely $\Sigma_xX=S_x\ast \nu_x$. Equivalently, $\Sigma_xX$ is (isometrically) the iterated spherical suspension over $\nu_x$.  In particular, we can identify $\nu_x$ as the $k$-fold iterated space of directions.  By \cite[Corollary 1.4]{Kapovitchlimits} we then understand that $\nu_x X$ is a topological sphere.\\ 

At this stage we can apply the slice theorem for compact isometric group actions on Alexandrov spaces from \cite{HarveySearle} to conclude that for $r>0$ sufficiently small there exists a $G$-equivariant homeomorphism between $G\times C(\nu_x)$ and $B_r(G\cdot x)$, where the action of $G$ on $G\times C(\nu_x)$ is by left multiplication on the first factor and $C(\nu_x)$ denotes the cone over the normal space $\nu_x$.  Hence $\pi(x)\in X/G$ has a neighborhood homeomorphic to $C(\nu_x)$. As we proved above that $\nu_x$ is a topological sphere, this shows that $X/G$ is a topological manifold.
\end{proof}

\vspace{.5cm}

\section{Proof of the main result}\label{sec:proof}

Our goal in this section is to prove \autoref{thm:torusstabilitysecintro}, which we restate below for the ease of readability.

\begin{theorem}\label{thm:torusstabilitysec}
Let $(M^n_i, g_i) \xrightarrow{GH} (X^k,\dist_X)$ satisfy ${\rm Sec}_{g_i} \ge -1$,  ${\rm diam}\, (M_i)\le D$, and such that the $M^n_i$ are all homeomorphic to tori $T^n$.  Then $X^k$ is homeomorphic to the $k$-dimensional torus $T^k$ for some $0\le k\le n$.
\end{theorem}

Let us open by outlining the proof strategy.  Consider the sequence of universal covers $(\tilde M_i, \tilde g_i, \tilde p_i)$, where $\tilde p_i \in \tilde M_i$ is any given reference point. We let $\Gamma_i = \pi_1(M_i)$ be the group of deck transformations acting by isometries on $\tilde M_i$, i.e. $\tilde M_i/\Gamma_i = M_i$. The proof of \autoref{thm:torusstabilitysec} will be divided into four steps:

\begin{itemize}
	\item[(1)] We build subgroups $\Lambda_i < \Gamma_i$ with the property that $\tilde M_i/\Lambda_i$ is homeomorphic to a torus with bounded diameter $\diam(\tilde M_i/\Lambda_i) \le 6^nD$, and is uniformly volume non-collapsing ${\rm Vol}(\tilde M_i/\Lambda_i) \ge c(n,D)$.

	\item[(2)] We apply Perelman's stability theorem to deduce that any limit of $\tilde M_i/\Lambda_i$ is homeomorphic to a torus $T^n$.  This will be used to show that  $(\tilde{M}_i,\tilde{g}_i, \tilde p_i) \to (\tilde X, \dist_{\tilde X}, \tilde x)$ in the pGH topology, where $(\tilde{X},\dist_{\tilde{X}})$ is an Alexandrov space with ${\rm Sec}\ge -1$ which is homeomorphic to $\setR^n$. Moreover, the homeomorphism can be chosen to be equivariant with respect to suitable $\mathbb{Z}^n$ actions.
	
	\item[(3)] We study the equivariant pGH limit $(\tilde M_i, \tilde p_i, \Gamma_i) \to (\tilde X, \tilde x, \Gamma)$ and prove that $\Gamma\equiv \mathbb{Z}^k \oplus \mathbb{R}^{n-k}$ for some $k\le n$ with $\Gamma$ acting freely.
	The key point consists in showing that $\Gamma$ does not have compact subgroups, and will be achieved by means of  \autoref{thm:boundedorbitstrivial}. We remark that we will endow $\Gamma$ and all its subgroups with the compact-open topology, as it is customary in this setting.

	\item[(4)] In the last step, we are going to rely on \autoref{thm:actionsontori} to see that $(X,\dist)$ is a homotopy $k$-torus, and then conclude it is homeomorphic to $T^k$.
\end{itemize}

\vspace{.3cm}

\subsection{Step 1}\label{sec:Step1}
Let us apply \autoref{thm:less dense} to $(M_i^n, g_i)$ in order to find 
\begin{equation}
\Lambda_i = <\gamma^i_1, \ldots , \gamma^i_n> 
\end{equation}
such that $\Lambda_i \equiv \mathbb{Z}^n$, $\Lambda_i < \Gamma_i$, 
$\diam(\tilde M_i/\Lambda_i) \le 6^nD$, and 
\begin{equation}\label{z}
	A(n)^{-1} D \| \gamma \|_1 \le \dist_{\tilde M_i}(x,\gamma(x)) \le A(n) D \| \gamma \|_1 \, , \quad \forall \, \gamma \in \Lambda_i\, , x\in \tilde M_i \, .
\end{equation} 
We claim that
\begin{itemize}
	\item[(a)] $A(n)^{-1}D \le \dist_{\tilde M_i^n}(\tilde p_i, \gamma^i_a(\tilde p_i)) \le A(n)D$ for every $a=1,\ldots,n$.
	
	\item[(b)] ${\rm Vol}(\tilde M_i/\Lambda_i) \ge c(n,D)$.
\end{itemize}
The first item immediately follows from \eqref{z} and the identity $\|\gamma^i_a\|_1 = 1$, which is simply from the definition of $\ell^1$ norm.
To check (b), we observe that \eqref{z} implies 
\begin{equation}
	\dist_{\tilde M_i}(\tilde p_i, \gamma(\tilde p_i) ) \ge A(n)^{-1} D \, , \quad \text{for every $\gamma\in \Lambda_i \setminus\{0\}$}\, .
\end{equation}
In particular,  $\pi_i: \tilde M_i \to \tilde M_i/\Lambda_i$ restricts to an isometry on the ball $B_{A(n)^{-1}D/2}(\tilde{p}_i)$. By \autoref{thm:covernoncollapsed} and Bishop-Gromov inequality, the balls $B_{A(n)^{-1}D/2}(\tilde{p}_i)$ are uniformly non-collapsed. 

\vspace{.3cm}

\subsection{Step 2}\label{sec:step2}

The spaces $(\tilde M_i/\Lambda_i, \tilde g_i)$ have uniformly bounded diameters and they are volume noncollapsing. Hence, up to the extraction of a subsequence, we can assume that $(\tilde M_i, \tilde g_i, \tilde p_i, \Lambda_i) \to (\tilde X, \dist_{\tilde X}, \tilde x, \Lambda)$ in the equivariant pGH topology, and 
$(\tilde M_i/\Lambda_i, \tilde g_i) \to (\tilde X/\Lambda,\dist_{\tilde X/\Lambda})$ in the GH topology, where $(\tilde X/\Lambda,\dist_{\tilde X/\Lambda})$ is $n$-dimensional Alexandrov space with ${\rm Sec}\ge -1$.\\

Recall now that our base assumption is that the $M^n_i$ are homeomorphic tori, and consequently $\tilde M_i/\Lambda_i$ are homeomorphic tori.  We then have by Perelman's stability theorem (see \cite{Perelmanstability,Perelmanstabilitysimpl} and the exposition in \cite{Kapovitchperelman}), that there exists a homeomorphism $\phi:T^n\to \tilde X/\Lambda$.\\

By (a) in \autoref{sec:Step1}, it is elementary to check that $\Lambda \equiv \mathbb{Z}^n$, and it acts freely on $\tilde X$. In particular, $\pi: \tilde{X} \to \tilde X/\Lambda$ is a covering map with covering group $\mathbb{Z}^n$.  On the other hand $\tilde X/\Lambda$ is homeomorphic to $T^n$, hence the Galois correspondence for covering spaces implies that $\tilde{X}$ is the universal cover of $\tilde X/\Lambda$.
Therefore we can lift the homeomorphism $\phi:T^n\to \tilde X/\Lambda$ to an equivariant homeomorphism $\Phi: (\setR^n, \mathbb{Z}^n)\to (\tilde{X},\Lambda)$, where $\mathbb{Z}^n$ is the standard integer lattice on $\setR^n$.

\subsection{Step 3}
Up to extracting another subsequence we have the convergence $(\tilde{M}_i, \tilde g_i,\tilde{p}_i,\Gamma_i) \to (\tilde{X}, \dist_{\tilde X} ,\tilde{x},\Gamma)$ in the equivariant pGH topology, where $\Gamma$ is an abelian Lie group acting by isometries on $\tilde{X}$.
We show that $\Gamma \equiv \mathbb{Z}^{k}\oplus\setR^{n-k}$,  and that $\Gamma$ acts freely on $\tilde X$.
We proceed in three steps.\\

	{\bf Claim 3.1:} There are no non-trivial compact subgroups $\Gamma'<\Gamma$.\\

	Let us take $\Gamma'<\Gamma$ to be a compact subgroup.  Let $\Phi:(\setR^n, \mathbb{Z}^n) \to (\tilde{X}, \Lambda)$ be the equivariant homeomorphism obtained at the end of Step 2.
	We induce a $\Gamma'$ action by homeomorphisms on $\setR^n$ via
	\begin{equation}
		\gamma'\cdot v:=\Phi^{-1}(\gamma'\cdot \Phi(v))\, , \quad
		\gamma'\in \Gamma'\, , \, \, v\in \setR^n\, .
	\end{equation}
	By the equivariance of $\Phi:(\setR^n, \mathbb{Z}^n) \to (\tilde{X}, \Lambda)$ and the fact that $\Lambda$ and $\Gamma'$ commute, the induced $\Gamma'$ action on $\setR^n$ commutes with the $\mathbb{Z}^n$ action.  Consequently its orbits are uniformly bounded with respect to any norm on $\setR^n$.  By \autoref{thm:boundedorbitstrivial} the induced $\Gamma'$ action is trivial and hence $\Gamma'=\{e\}$ as claimed.  $\qed$\\

	{\bf Claim 3.2:} The group $\Gamma$ acts freely on $\tilde X$.\\

	Let $x\in \tilde X$ and consider the isotropy subgroup $\Gamma_x:=\{\gamma\in \Gamma\, : \gamma(x) = x\}<\Gamma$. It is immediate to check that $\Gamma_x$ is compact (with respect to the compact-open topology). On the other hand, we have shown in Claim 3.1 that $\Gamma$ does not admit compact subgroups. $\qed$\\

	{\bf Claim 3.3:} $\Gamma \equiv \mathbb{Z}^{k} \oplus \mathbb{R}^{n-k}$ for some $k\in \{0,\dots,n\}$.\\

	It is well-known that abelian Lie groups split as
	\begin{equation}
		\Gamma \equiv \Gamma/\Gamma_0 \oplus \Gamma_0 \, ,
	\end{equation}
	where $\Gamma_0<\Gamma$ is the connected component of the identity, and $\Gamma/\Gamma_0 $ is discrete. 
	Moreover, $\Gamma_0$ must be isomorphic to $T^a \oplus \mathbb{R}^b$, where $a$ and $b$ are nonnegative integers, as it is a connected abelian Lie group. However, using Claim 3.1 we have that $\Gamma$ has no compact subgroups, and hence we have that $\Gamma_0 = \setR^b$.\\
	
    On the other hand let us consider the abelian group $\Gamma/\Gamma_0$.  Arguing as in Claim 3.2 we deduce that $\Gamma/\Gamma_0$ acts discretely and freely on $\tilde X/\Gamma_0$. Moreover, it is generated by isometries displacing any reference point less than $2D$.     
   Consequently, $\Gamma/\Gamma_0$ is  finitely generated, so that $\Gamma/\Gamma_0 = \setZ^c\oplus \bigoplus \setZ_{c_j}$.   Using Claim 3.1 we have that $\Gamma/\Gamma_0$ is torsion free, and hence $\Gamma/\Gamma_0 \equiv \mathbb{Z}^c$.\\ 
	
	Since $\Gamma \equiv \mathbb{Z}^c \oplus \mathbb{R}^b$ acts cocompactly and freely, we can extract a lattice $\Lambda' \equiv \mathbb{Z}^c \oplus \setZ^b < \mathbb{Z}^c \oplus \mathbb{R}^b \equiv \Gamma$ acting cocompactly on $\tilde X \cong \mathbb{R}^n$. The quotient $\tilde X/\Lambda'$ is a compact aspherical manifold with abelian fundamental group.   By a classical argument, see for instance \cite{Whitehead}, it is homotopy equivalent to $T^n$.  In particular $\setZ^n = \pi_1(T^n)=\setZ^c\oplus\setZ^b$ and hence $c + b = n$.  This finishes the proof of Claim 3.3 and hence Step 3. $\qed$\\

\subsection{Step 4} We are now in position to conclude the proof of \autoref{thm:torusstabilitysec}.

\medskip

Let $\Lambda'<\Gamma \equiv \mathbb{Z}^{k}\oplus \mathbb{R}^{n-k}$ be the sub-lattice from the end of Claim 3.3, so that $\Lambda' \equiv \mathbb{Z}^{k} \oplus \mathbb{Z}^{n-k}$ with  $\Gamma/\Lambda' \equiv T^{n-k}$.  As in Claim 3.3 we know that $\tilde X/\Lambda'$ is a homotopy torus, and additionally we know that $\tilde X/\Lambda'$ is a topological manifold as $\tilde X\equiv \setR^n$.  By \cite{HsiangWall} for dimensions greater or equal than $5$, \cite{FreedmanQuinn} in dimension $4$, the solution of the geometrization conjecture \cite{PerelmanPoincareI,PerelmanPoincareII,PerelmanPoincareIII} in dimension $3$, and the classification of surfaces in dimension $2$, we have that $\tilde X/\Lambda'$ is homeomorphic to $T^n$ \footnote{We could avoid this heavy use of machinery by instead letting $\Lambda'$ be a refinement of $\Lambda$ from Step 2, where we have already deduced that $\tilde X/\Lambda$ is a torus.  However, this argument appears again at the end of the proof as well, and the heavy machinery is not avoidable there.}.\\

Thus $\tilde X/\Lambda'$ is homeomorphic to the torus $T^n$ and there is an induced $\setR^{n-k}/\Lambda'\equiv T^{n-k}$ action by isometries.   The quotient space $(\tilde X/\Lambda')/T^{n-k}$ coincides with the limit space $X$.\\

By \autoref{thm:actionsontori}, this $T^{n-k}$ action is free and the projection to the quotient is a trivial principal $T^{n-k}$ fibration. In particular, $X\times T^{n-k}$ is homeomorphic to $\tilde X/\Lambda'\equiv T^n$.   
Since $\tilde{X}/\Lambda'$ is locally smoothable, $X$ is itself a topological manifold by \autoref{prop:quotienttopman}.

As $X\times T^{n-k}$ is homeomorphic to $T^n$, we show that $X$ is aspherical and $\pi_1(X)=\mathbb{Z}^{k}$: 
\begin{equation}
	\begin{split}
		&0 = \pi_j(T^n) = \pi_j(X \times T^{n-k}) = \pi_j(X) \oplus \pi_j(T^{n-k}) = \pi_j(X) \quad j\ge 2 \, ,
		\\
		&\mathbb{Z}^n = \pi_1(T^n) = \pi_1(X \times T^{n-k}) = \pi_1(X) \oplus \mathbb Z^{n-k}
		\implies \pi_1(X) = \mathbb Z^k \, .
	\end{split}
\end{equation}	
Thus, by \cite{Whitehead}, $X$ is a topological manifold which is homotopy equivalent to a $k$-torus.  Arguing as before, we have by \cite{HsiangWall} for dimensions greater or equal than $5$, \cite{FreedmanQuinn} in dimension $4$, the solution of the geometrization conjecture \cite{PerelmanPoincareI,PerelmanPoincareII,PerelmanPoincareIII} in dimension $3$, and the classification of surfaces in dimension $2$ that $X$ is homeomorphic to $T^k$.  \\

\subsection{Outline of proof of \autoref{thm:n=3}}\label{ss:Ricci_lower_bound_n=3}

We can follow the same approach as in \autoref{thm:torusstabilitysecintro} with two major modifications.  More precisely, given a sequence of tori $(M^3_i,g_i)$ with $\Ric_{g_i}\ge -2$ and $\diam(M_i) \le D$, we first find intermediate noncollapsing coverings $(\tilde M_i/\Lambda_i, \tilde g_i)$ as in Step 1 of proof of \autoref{thm:torusstabilitysecintro}. Notice that the tools required for the selection of these covers, namely \autoref{thm:covernoncollapsed} and  \autoref{thm:less dense} below, do not depend on the lower sectional curvature bound.

In Step 2 of the proof, we replace Perelman's stability theorem with \cite[Theorem 1.7]{Simon}.
This allows to conclude that the limit $(\tilde M_i, \tilde g_i, \tilde p_i, \Lambda_i) \xrightarrow{pGH} (\tilde X, \dist_{\tilde X}, \tilde x, \Lambda)$ is equivariantly homeomorphic to $(\mathbb{R}^3, \mathbb{Z}^3)$.

Step 3 does not require any change with respect to the case of lower sectional curvature bounds.

In Step 4, in order to infer that the quotient of a topological $T^3$ with respect to a free torus action by homeomorphisms is a topological manifold, we can rely on \cite{Bredon2} (see also \cite{Raymond}). The rest of the argument does not require any modification.\\

\subsection{Stability of tori in the Alexandrov setting}\label{subsec:counterAlexandrov}

In this section, we consider $(M^n_i,\dist_i) \to (X^k,\dist_X)$ where $(M_i^n,\dist_i)$ are $n$-dimensional Alexandrov spaces with ${\rm curv}\ge -1$ and $\diam (M_i) \le D$. 

As anticipated in the introduction, if $M_i^n$ are all homeomorphic to tori $T^n$, then $X^k$ is homotopy equivalent to $T^k$. This follows from the fact that $X^k \times T^{n-k}$ is homeomorphic to $T^n$, as a consequence of our proof of \autoref{thm:torusstabilitysecintro}, and that $X$ is a CW-complex.

\subsubsection{Failure of the topological stability}
We now discuss an example of a sequence $(T^5,\dist_i)\xrightarrow{GH} (Y^4,\dist)$ of Alexandrov spaces $(T^5,\dist_i)$ homeomorphic to the torus $T^5$ and converging under a uniform lower curvature bound to an Alexandrov space $(Y^4,\dist)$ which is not homeomorphic to $T^4$. The example was indicated to the authors by Vitali Kapovitch. Its effect is to show that \autoref{thm:torusstabilitysecintro} does not generalize to Alexandrov spaces. \\

The main idea for the construction is borrowed from \cite{Kapovitchlimits}, see in particular Example 1.5, Corollary 1.6 and its proof therein.\\

Let $\Sigma^3$ be the Poincar\'e homology sphere with the metric of constant curvature $1$. Recall that it can be constructed as a quotient of $S^3$ endowed with its standard metric of constant curvature $1$ with respect to a free isometric action of the icosahedral group.

The spherical suspension $S^1\Sigma^3$ over $\Sigma^3$  is a $4$-dimensional Alexandrov space with $\mathrm{curv}\ge 1$ which is homotopy equivalent to $S^4$. However, it is not homeomorphic to $S^4$. Indeed, it is not a topological manifold as the two suspension points are not manifold points.

The distance on $S^1\Sigma^3$ is induced by a smooth Riemannian metric with $\mathrm{curv}\ge 1$ away from the two suspension points. In particular, we can fix a non-singular point $x\in S^1\Sigma^3$ and find a small neighbourhood $U\ni x$ isometric to a smooth Riemannian manifold. Then we can perform the connected sum $S^1\Sigma^3\#T^4$ on $U$ with a distance $\dist$ so that $(S^1\Sigma^3\#T^4,\dist)$ has curvature bounded below by some $K\le 0$ in the Alexandrov sense and such that $S^1\Sigma^3\setminus U$ embeds isometrically in $S^1\Sigma^3\#T^4$.\\

Notice that $S^1\Sigma^3\#T^4$ is homotopy equivalent to $T^4$. However, it is not homeomorphic to $T^4$. Indeed, it is not a topological manifold.  On the other hand, we claim that $(S^1\Sigma^3\#T^4)\times T^1$ is homeomorphic to $T^5$.\\

In order to establish the claim, let us first notice that $(S^1\Sigma^3\#T^4)\times T^1$ is homotopy equivalent to $T^5$, because $S^1\Sigma^3\#T^4$ is homotopy equivalent to $T^4$, as remarked above.  Also note that the product metric on $(S^1\Sigma^3\#T^4)\times T^1_{\epsilon}$ has curvature bounded from below by $K$ for any $\epsilon>0$.  Now let us first see that $(S^1\Sigma^3\#T^4)\times T^1_{\epsilon}$ is a topological manifold.  At any (metrically) singular point $y\in (S^1\Sigma^3\#T^4)\times T^1_{\epsilon}$, the tangent cone is isometric to $\setR\times C(\Sigma^3)=C(S^1\Sigma^3)$. By Edwards' double suspension theorem (cf. \cite{Daverman}), $C(S^1\Sigma^3)$ is homeomorphic to $\setR^5$. Hence, by Perelman's conical neighbourhood theorem, the metrically singular points of $(S^1\Sigma^3\#T^4)\times T^1_{\epsilon}$ are also manifold points.  Therefore, $(S^1\Sigma^3\#T^4)\times T^1$ is a topological manifold which is homotopy equivalent to $T^5$. Hence by \cite{HsiangWall}, it is homeomorphic to $T^5$.  \\ 

The sequence $\left((S^1\Sigma^3\#T^4)\times T^1_{1/i},\dist_i\right)_{i\in\setN}$, where $\dist_i$ is the natural product distance, collapses to $\left((S^1\Sigma^3\#T^4),\dist\right)$ with curvature bounded from below by $K\le 0$ and diameter bounded from above.  However, $(S^1\Sigma^3\#T^4)\times T^1$ is homeomorphic to $T^5$, while $S^1\Sigma^3\#T^4$ is not homeomorphic to $T^4$.  $\qed$\\

\subsubsection{Stability in dimension $n\le 4$}

	\label{rm:optimaldimension}
We remark that $n=5$ is the lowest possible dimension for a counterexample to the full torus stability \autoref{thm:torusstabilitysecintro} in the Alexandrov case. We discuss here only the case $n=4$ and $k=3$. The other cases can be treated with some (simpler) variants of the argument below.

If $(T^4,\dist_i)\xrightarrow{GH}(X^3,\dist_X)$ are Alexandrov tori converging under a uniform lower curvature bound and a uniform upper diameter bound, borrowing the notation from the proof of the main theorem, we can find a non-collapsed limit $\tilde{X}^4/\Lambda'$ of intermediate coverings of $(T^4,\dist_i)$ such that $X^3=(\tilde{X}^4/\Lambda')/T^1$, with $T^1$ acting freely and by isometries.

The argument in the proof of \autoref{prop:quotienttopman} shows that all tangent cones of $\tilde{X}^4$ split a factor $\setR$ isometrically. Moreover, $\tilde{X}^4/\Lambda'$ is homeomorphic to $T^4$ by Perelman's stability theorem. This is sufficient to show that any tangent cone of $\tilde{X}^4$ is isometric to $\setR\times C(S^2)$, for some Alexandrov metric with curvature $\ge 1$ on $S^2$. Indeed, the only other option would be that the tangent cone is isometric to $\setR\times C(\mathbb{RP}^2)$ for some Alexandrov metric with curvature $\ge 1$ on $\mathbb{RP}^2$. However, this would be in contradiction with the topological regularity of $\tilde{X}^4$ by Perelman's conical neighbourhood theorem.

Following again the argument in the proof of \autoref{prop:quotienttopman}, it follows that the tangent cones of $(X^3,\dist_X)$ are all homeomorphic to $\setR^3$. Hence, by Perelman's conical neighbourhood theorem again, $X^3$ is a topological manifold. Combined with the fact that it is a homotopy $3$-torus, this is sufficient to show that $X^3$ is homeomorphic to $T^3$.


\begin{thebibliography}{GMS13}




\bibitem{Bredon2}
\textsc{Bredon, G. E.,}
\textit{Some theorems on transformation groups.}
Ann. Math. (2) {\bf 67}, 104-118 (1958).


\bibitem{Bredon}
\textsc{Bredon, G. E.,} 
\textit{Introduction to compact transformation groups.} 
Pure and Applied Mathematics, Vol. 46. Academic Press, New York-London, 1972. xiii+459 pp.




\bibitem{CheegerGromovI}
\textsc{Cheeger, J.; Gromov, M.,} 
\textit{Collapsing Riemannian manifolds while keeping their curvature bounded. I} 
J. Differential Geom. {\bf 23} (1986), no. 3, 309--346.




\bibitem{CheegerGromovII}
\textsc{Cheeger, J.; Gromov, M.,} 
\textit{Collapsing Riemannian manifolds while keeping their curvature bounded. II} 
J. Differential Geom. {\bf 32} (1990), no. 1, 269--298.


\bibitem{CheegerFukayaGromov}
\textsc{Cheeger, J.; Fukaya, K.; Gromov, M.,} 
\textit{Nilpotent structures and invariant metrics on collapsed manifolds.} 
J. Amer. Math. Soc. {\bf 5} (1992), no. 2, 327--372.


\bibitem{Colding}
\textsc{Colding, T. H.:}
\textit{Ricci curvature and volume convergence.}
Ann. Math. (2) {\bf 145}, No. 3, 477-501 (1997).

\bibitem{ConnerMontgomery}
\textsc{Conner, P. E.; Montgomery, D.,} 
\textit{Transformation groups on a $K(\pi,1)$. I.} 
Michigan Math. J. {\bf 6} (1959), 405–412.

\bibitem{ConnerRaymond}
\textsc{Conner, P. E.; Raymond, F.,} 
\textit{Actions of compact Lie groups on aspherical manifolds.} 
1970 Topology of Manifolds (Proc. Inst., Univ. of Georgia, Athens, Ga., 1969) pp. 227–264 Markham, Chicago, Ill.



\bibitem{Daverman}
\textsc{Daverman, R.~J.,}
\textit{Decompositions of manifolds.}
Pure and Applied Mathematics, 124. Orlando etc.: Academic Press, XI, 317 p.; (1986).


\bibitem{FreedmanQuinn}
\textsc{Freedman, M.-H.; Quinn, F.,} 
\textit{Topology of 4-manifolds.} 
Princeton Mathematical Series, {\bf 39}. Princeton University Press, Princeton, NJ, 1990. viii+259 pp.




\bibitem{FukayacollapsingI}
\textsc{Fukaya, K.,} 
\textit{Collapsing Riemannian manifolds to ones of lower dimensions.} 
J. Differential Geom. {\bf 25} (1987), no. 1, 139--156.


\bibitem{Fukayaboundary}
\textsc{Fukaya, K.,} 
\textit{ A boundary of the set of the Riemannian manifolds with bounded curvatures and diameters.} 
J. Differential Geom. {\bf 28} (1988), no. 1, 1–21





\bibitem{FukayacollapsingII}
\textsc{Fukaya, K.,} 
\textit{Collapsing Riemannian manifolds to ones of lower dimensions. II.} 
J. Math. Soc. Japan {\bf 41} (1989), no. 2, 333--356.




\bibitem{Fukayaorbifold}
\textsc{Fukaya, K.,} 
\textit{A compactness theorem of a set of aspherical Riemannian orbifolds.} 
A fête of topology, 391–413, Academic Press, Boston, MA, 1988.


\bibitem{GalazSearle}
\textsc{Galaz-Garcia, F.; Searle, C.,} 
\textit{Cohomogeneity one Alexandrov spaces.} 
Transform. Groups {\bf 16} (2011), no. 1, 91–107.


\bibitem{Gromovalmostflat}
\textsc{Gromov, M.,} 
\textit{Almost flat manifolds.} 
J. Differential Geometry {\bf 13} (1978), no. 2, 231–241.




\bibitem{Gromovfilling}
\textsc{Gromov, M.,} 
\textit{Filling Riemannian manifolds.} 
J. Differential Geom. {\bf 18} (1983), no. 1, 1–147. 


\bibitem{Gromovbook}
\textsc{Gromov, M.,} 
\textit{Metric structures for Riemannian and non-Riemannian spaces.} 
Based on the 1981 French original. With appendices by M. Katz, P. Pansu and S. Semmes. 
Modern Birkhäuser Classics. Birkhäuser Boston, Inc., Boston, MA, 2007. xx+585 pp. 





\bibitem{GromovLawson}
\textsc{Gromov,  M.; Lawson Jr, H.-B.,} 
\textit{Spin and scalar curvature in the presence of a fundamental
group I.} 
Ann. of Math., {\bf 111} (1980), 209--230.




\bibitem{Guth}
\textsc{Guth, L.,} 
\textit{Volumes of balls in large Riemannian manifolds.} 
Ann. of Math. (2) {\bf 173} (2011), no. 1, 51–76.


\bibitem{HarveySearle}
\textsc{Harvey, J.; Searle, C.,} 
\textit{Orientation and symmetries of Alexandrov spaces with applications in positive curvature.} 
J. Geom. Anal. {\bf 27} (2017), no. 2, 1636–1666.



\bibitem{HsiangWall}
\textsc{Hsiang, W. C.; Wall, C. T.  C.,} 
\textit{On homotopy tori. II.} 
Bull. London Math. Soc. {\bf 1} (1969), 341–342. 

\bibitem{KS20}
\textsc{Kloeckner, B.; Sabourau, S.,}
\textit{Mixed sectional-Ricci curvature obstructions on tori}, 
J. Topol. Anal. {\bf 12} (2020), no.3, 713--734.


\bibitem{Kapovitchlimits}
\textsc{Kapovitch, V.,} 
\textit{Regularity of limits of noncollapsing sequences of manifolds.} 
Geom. Funct. Anal. {\bf 12} (2002), no. 1, 121–137.




\bibitem{Kapovitchcollapse}
\textsc{Kapovitch, V.,}
\textit{Restrictions on collapsing with a lower sectional curvature bound.}
Math. Z. {\bf 249} (2005), no. 3, 519-539.





\bibitem{Kapovitchperelman}
\textsc{Kapovitch, V.,} 
\textit{Perelman's stability theorem.} 
Surveys in differential geometry. Vol. XI, 103–136, Surv. Differ. Geom., {\bf 11}, Int. Press, Somerville, MA, 2007.






\bibitem{Katz20}
\textsc{Katz, M. G.,}
\textit{Torus cannot collapse to a segment.}
J. Geom., {\bf 111} (2020), Paper No. 13, 8 pp.



\bibitem{LebedevaPetrunin}
\textsc{Lebedeva, N., Petrunin, A.,}
\textit{Curvature tensor of smoothable Alexandrov spaces.}
To appear on Geom. Topol. Preprint 	arXiv:2202.13420 (2022). 



\bibitem{Mahler46}
\textsc{Mahler, K.,}
\textit{On lattice points in n-dimensional star bodies. I: Existence theorems.}
Proc. R. Soc. Lond., Ser. A {\bf 187}, 151-187 (1946).







\bibitem{Perelmanstability}
\textsc{Perelman, G. Ya.,}
\textit{Alexandrov spaces with curvatures bounded from below. II}, 
preprint (1991).


\bibitem{Perelmanstabilitysimpl}
\textsc{Perelman, G. Ya.,} 
\textit{Elements of Morse theory on Aleksandrov spaces.} 
(Russian) Algebra i Analiz {\bf 5} (1993), no. 1, 232–241; translation in St. Petersburg Math. J. {\bf 5} (1994), no. 1, 205–213






\bibitem{PerelmanPoincareI}
\textsc{Perelman, G. Ya.,} 
\textit{The entropy formula for the Ricci flow and its geometric applications.} 
Preprint, 2002, math.DG/0211159


\bibitem{PerelmanPoincareII}
\textsc{Perelman, G. Ya.,} 
\textit{Ricci flow with surgery on three-manifolds.} 
Preprint, 2003, math.DG/0303109.


\bibitem{PerelmanPoincareIII}
\textsc{Perelman, G. Ya.,} 
\textit{Finite extinction time for the solutions to the Ricci flow on certain three-manifolds.} Preprint, 2003, math.DG/0307245.


\bibitem{Raymond}
\textsc{Raymond, F.,}
\textit{Classification of the actions of the circle on 3-manifolds.}
Trans. Am. Math. Soc. {\bf 131}, 51-78 (1968).

\bibitem{SchoenYau}
\textsc{Schoen, R.; Yau, S.-T.:}
\textit{On the structure of manifolds with positive scalar curvature.} 
Manuscr. Math. {\bf 28}, 159-183 (1979).


\bibitem{ShioyaYamaguchi}
\textsc{Shioya, T.; Yamaguchi, T.:}
\textit{Collapsing three-manifolds under a lower curvature bound.}
J. Differ. Geom. {\bf 56}, No. 1, 1-66 (2000).



\bibitem{Simon}
\textsc{Simon, M.;}
\textit{Ricci flow of non-collapsed three manifolds whose Ricci curvature is bounded from below.}
J. Reine Angew. Math. {\bf 662} (2012), 59--94.




\bibitem{Whitehead}
\textsc{Whitehead, G. W.,} 
\textit{Elements of homotopy theory.} 
Graduate Texts in Mathematics, {\bf 61}. Springer-Verlag, New York-Berlin, 1978. xxi+744 pp.




\bibitem{Yamaguchicoll}
\textsc{Yamaguchi, T.,} 
\textit{Collapsing and pinching under a lower curvature bound.} 
Ann. of Math., {\bf 133}
(1991), 317--357.




\bibitem{Zamora}
\textsc{Zamora, S.,} 
\textit{Tori can't collapse to an interval} 
Electron. Res. Arch. {\bf 29} (2021), no. 4, 2637--2644.


\bibitem{Zamora2}
\textsc{Zamora, S.,} 
\textit{First Betti number and collapse,} 
preprint (2022), arXiv:2209.12628v1.


\end{thebibliography}
\end{document}